\documentclass[11pt]{article}
\usepackage[margin=1in]{geometry}

\usepackage{graphicx} 
\usepackage{amsmath} 
\usepackage{amssymb}  
\usepackage{amsthm}

\usepackage{color}
\allowdisplaybreaks[1]

\usepackage{algorithm}
\usepackage{algpseudocode}
\usepackage{hyperref} 

\theoremstyle{plain}  
\newtheorem{theorem}{Theorem}[section]

\newtheorem{prop}[theorem]{Proposition}
\newtheorem{cor}[theorem]{Corollary}
\newtheorem{defn}[theorem]{Definition}

\theoremstyle{definition}

\theoremstyle{remark} 

\newtheorem{remark}{Remark}[section]

\def\0{{\bf 0}}
\def\1{{\bf 1}}

\def \bmat{\left[\begin{matrix}}
	\def \emat{\end{matrix}\right]}

\def \xy1vec{\left[\begin{matrix}x\\y\\1\end{matrix}\right]}
\def \QED{\begin{flushright}\Halmos\end{flushright}\end{proof}}
\def \defeq{\mathrel{\mathop{:}}=}

\long\def\old#1{}

\definecolor{DarkerGreen}{RGB}{0,170,0}

\long\def\jeff#1{{\color{black}#1}}
\definecolor{orange}{rgb}{1,0.5,0}

\title{\LARGE \bf On the Complexity of Testing Attainment of the Optimal Value in Nonlinear Optimization}
\author{Amir Ali Ahmadi and Jeffrey Zhang \thanks{The authors are with  the department of Operations Research and Financial Engineering at Princeton University. Email: \{\texttt{a\_a\_a}, \texttt{jeffz}\}\texttt{@princeton.edu}. 
 This work was partially supported by the DARPA Young Faculty Award, the CAREER Award of the NSF, the Innovation Award of the School of Engineering and Applied Sciences at Princeton University, the MURI award of the AFOSR, the Google Faculty Award, and the Sloan Fellowship.}}
\begin{document}
\date{}
\maketitle


\begin{abstract}
	\noindent
	We prove that unless P=NP, there exists no polynomial time (or even pseudo-polynomial time) algorithm that can test whether the optimal value of a nonlinear optimization problem where the objective and constraints are given by low-degree polynomials is attained. If the degrees of these polynomials are fixed, our results along with previously-known ``Frank-Wolfe type'' theorems imply that exactly one of two cases can occur: either the optimal value is attained on every instance, or it is strongly NP-hard to distinguish attainment from non-attainment. We also show that testing for some well-known sufficient conditions for attainment of the optimal value, such as coercivity of the objective function and closedness and boundedness of the feasible set, is strongly NP-hard. As a byproduct, our proofs imply that testing the Archimedean property of a quadratic module is strongly NP-hard, a property that is of independent interest to the convergence of the Lasserre hierarchy. Finally, we give semidefinite programming (SDP)-based sufficient conditions for attainment of the optimal value, in particular a new characterization of coercive polynomials that lends itself to an SDP hierarchy.

\end{abstract}


\section{Introduction}\label{Sec: intro}

Consider an optimization problem of the form

\begin{equation}\label{Defn: OPT}
\begin{aligned}
& \underset{x}{\inf}
& & f(x) \\
& \text{subject to}
&& x \in \Omega,\\
\end{aligned}
\end{equation}
where $f: \mathbb{R}^n \to \mathbb{R}$ and $\Omega \subseteq \mathbb{R}^n$. In this paper, we are interested in the complexity of checking whether we can replace the ``$\inf$'' with a ``$\min$''. More precisely, suppose the optimal value $f^*$ of this problem is finite, i.e., the problem is feasible and bounded below. We would like to test if the optimal value is \emph{attained}, i.e., whether there exists a point $x^* \in \Omega$ such that $f(x^*) \le f(x)\ \forall x \in \Omega$, or equivalently, such that $f^*=f(x^*).$ Such a point $x^*$ will be termed an \emph{optimal solution}. 





Existence of optimal solutions is a fundamental question in optimization and its study has a long history, dating back to the nineteenth century with the extreme value theorem of Bolzano and Weierstrass. While the problem has been researched in depth from an analytical perspective, to our knowledge, it has not been studied from an algorithmic viewpoint. The most basic question in this direction is to ask whether one can efficiently check for existence of an optimal solution with an algorithm that scales reasonably with the description size of the function $f$ and the set $\Omega$ in (\ref{Defn: OPT}).\footnote{To remove possible confusion, we emphasize that our focus in this paper is not on the complexity of testing feasibility or unboundedness of problem (\ref{Defn: OPT}), which have already been studied extensively. On the contrary, all optimization problems that we consider are by construction feasible and bounded below.}

A class of optimization problems that allows for a rigorous study of this algorithmic question is the class of \emph{polynomial optimization problems (POPs)}. These are problems where one minimizes a polynomial function over a closed basic semialgebraic set, i.e., problems of the type

\begin{equation}\label{Defn: pop}
\begin{aligned}
& \underset{x}{\inf}
& & f(x) \\
& \text{subject to}
&& g_i(x) \ge 0, \forall i \in \{1,\ldots,m\},\\
\end{aligned}
\end{equation}
where $f,g_i$ are polynomial functions. The question of testing attainment of the optimal value for POPs has appeared in the literature explicitly. For example, Nie, Demmel, and Sturmfels describe an algorithm for globally solving an unconstrained POP which requires as an assumption that the optimal value be attained \cite{nie2006minimizing}. This leads them to make the following remark in their conclusion section:
\begin{quote}
	``This assumption is non-trivial, and we do not address the (important and difficult) question of how to verify that a given polynomial $f(x)$ has this property.''
\end{quote}

Prior literature on existence of optimal solutions to POPs has focused on identifying cases where existence is always guaranteed. The best-known result here is the case of linear programming (i.e., when the degrees of $f$ and $g_i$ are one). In this case, the optimal value of the problem is always attained. This result was extended by Frank and Wolfe to the case where $f$ is quadratic and the polynomials $g_i$ are linear~\cite{frank1956algorithm}. Consequently, results concerning attainment of the optimal value are sometimes referred to as ``Frank-Wolfe type'' theorems in the literature \cite{belousov2002frank,luo1999extensions}. Andronov et al. showed that the same statement holds again when $f$ is cubic (and the polynomials $g_i$ are linear)~\cite{andronov1982solvability}. 

Our results in this paper show that in all other cases, it is strongly NP-hard to determine whether a polynomial optimization problem attains its optimal value. This implies that unless P=NP, there is no polynomial-time (or even pseudo-polynomial time)  algorithm for checking this property. Nevertheless, it follows from the Tarski-Seidenberg quantifier elimination theory \cite{seidenberg1954new,tarski1951decision} that this problem is decidable, i.e., can be solved in finite time. \jeff{There are also probabilistic algorithms that test for attainment of the optimal value of a POP \cite{greuet2011deciding,greuet2014probabilistic}, but their complexities are exponential in the number of variables.}

In this paper, we also study the complexity of testing several well-known sufficient conditions for attainment of the optimal value (see Section \ref{SSec: Contributions} below). One sufficient condition that we do not consider but that is worth noting is for the polynomials $f,-g_1,\ldots,-g_m$ to all be convex (see \cite{belousov2002frank} for a proof, \cite{luo1999extensions} for the special case where $f$ and $g_i$ are quadratics, and \cite{bertsekas2007set} for other extensions). The reason we exclude this sufficient condition from our study is that the complexity of checking convexity of polynomials has already been analyzed in \cite{ahmadi2013np}.

\subsection{Organization and Contributions of the Paper}\label{SSec: Contributions}

As mentioned before, this paper concerns itself with the complexity of testing attainment of the optimal value of a polynomial optimization problem. More specifically, we show in Section \ref{Sec: AOS} that it is strongly NP-hard to test attainment when the objective function has degree 4, even in absence of any constraints (Theorem \ref{Thm: AOS NP-hard o4c0}), and when the constraints are of degree 2, even when the objective is linear (Theorem \ref{Thm: AOS NP-hard o1c2}). 
%

In Section \ref{Sec: Sufficient Conditions}, we show that several well-known sufficient conditions for attainment of the optimal value in a POP are also strongly NP-hard to test. These include coercivity of the objective function (Theorem \ref{Thm: Coercive NP-hard}), closedness of a bounded feasible set (Theorem \ref{Thm: Closedness} and Remark \ref{rem:boundedness.compactness}), boundedness of a closed feasible set (Corollary \ref{Thm: Boundedness}), a robust analogue of compactness known as stable compactness (Corollary \ref{Thm: Stable Compact NP-hard}), and an algebraic certificate of compactness known as the Archimedean property (Theorem \ref{Thm: Archimedean NP-hard}). The latter property is of independent interest to the convergence of the Lasserre hierarchy, as discussed in Section \ref{SSSec: Archimedean}.

In Section \ref{Sec: Algorithms}, we give semidefinite programming (SDP) based hierarchies for testing compactness of the feasible set and coercivity of the objective function of a POP (Propositions \ref{Prop: Compactness Stengle} and \ref{Prop: Coercive sublevel sets bounded}). The hierarchy for compactness comes from a straightforward application of Stengle's Positivstellensatz (cf. Theorem \ref{Thm: Stengle}), but the one for coercivity requires us to develop a new characterization of coercive polynomials (Theorem \ref{Thm: Polynomial Radius}). We end the paper in Section \ref{Sec:conclusion} with a summary and some brief concluding remarks.

\section{NP-hardness of Testing Attainment of the Optimal Value}\label{Sec: AOS}

In this section, we show that testing attainment of the optimal value of a polynomial optimization problem is NP-hard. Throughout this paper, when we study complexity questions around problem (\ref{Defn: pop}), we fix the degrees of all polynomials involved and think of the number of variables and the coefficients of these polynomials as input. Since we are working in the Turing model of computation, all the coefficients are rational numbers and the input size can be taken to be the total number of bits needed to represent the numerators and denominators of these coefficients.

Our proofs of hardness are based on reductions from ONE-IN-THREE 3SAT which is known to be NP-hard \cite{schaefer1978complexity}. Recall that in ONE-IN-THREE 3SAT, we are given a 3SAT instance (i.e., a collection of clauses, where each clause consists of exactly three literals, and each literal is either a variable or its negation) and we are asked to decide whether there exists a $\{0, 1\}$ assignment to the variables that makes the expression true with the additional property that each clause has \emph{exactly} one true literal.


\begin{theorem}\label{Thm: AOS NP-hard o4c0} Testing whether a degree-4 polynomial attains its unconstrained infimum is strongly\footnote{We recall that a strong NP-hardness result implies that the problem remains NP-hard even if the size (bit length) of the coefficients of the polynomial is $O(\log(n))$, where $n$ is the number of variables. For a strongly NP-hard problem, even a pseudo-polynomial time algorithm cannot exist unless P=NP. See \cite{garey2002computers} for precise definitions and more details.} NP-hard.
\end{theorem}

\begin{proof}
	
	Consider a ONE-IN-THREE 3SAT instance $\phi$ with $n$ variables $x_1,\ldots,x_n,$ and $k$ clauses. Let $s_\phi(x):\mathbb{R}^n\rightarrow \mathbb{R}$ be defined as
	\begin{equation}\label{eq: 1 in 3 sos}s_\phi(x) = \sum_{i=1}^k (\phi_{i1} + \phi_{i2} + \phi_{i3} + 1)^2 + \sum_{i=1}^n (1-x_i^2)^2,\end{equation}
	where $\phi_{it} = x_j$ if the $t$-th literal in the $i$-th clause is $x_j$, and $\phi_{it} = -x_j$ if it is $\neg x_j$ (i.e., the negation of $x_j$). Now, let 
	\jeff{\begin{equation}\label{eq: 1 in 3 sos plus unattained}
	p_\phi(x,y,z,\lambda) \mathrel{\mathop{:}}= \lambda^2s_\phi(x) + (1-\lambda)^2(y^2 + (yz-1)^2),\end{equation}}where $y,z,\lambda \in \mathbb{R}.$
	We show that $p_{\phi}$ achieves its infimum if and only if $\phi$ is satisfiable. Note that the reduction is polynomial in length and the coefficients of $p_\phi$ are \jeff{at most a constant factor of $n+k$ in absolute value}.
	
	If $\phi$ has a satisfying assignment, then for any $y$ and $z$, letting $\lambda = 1$, $x_i=1$ if the variable is true in that assignment and $x_i = -1$ if it is false, results in a zero of $p_\phi$. As $p_{\phi}$ is a sum of squares and hence nonnegative, we have shown that it achieves its infimum.
	
	Now suppose that $\phi$ is not satisfiable. We will show that $p_{\phi}$ is positive everywhere but gets arbitrarily close to zero. To see the latter claim, simply set $\lambda=0, z=\frac{1}{y}$, and let $y \rightarrow 0.$ To see the former claim, suppose for the sake of contradiction that $p_{\phi}$ has a zero. \jeff{Since} $y^2+(yz-1)^2$ is always positive, we must have $\lambda=1$ \jeff{in order for the second term to be zero.} \jeff{Then, in order for the whole expression to be zero, we must also have that} $s_{\phi}(x)$ must vanish at some $x.$ But any zero of $s_\phi$ must have each $x\in \{-1,1\}^n$, due to the second term of $s_\phi$. However, because the instance $\phi$ is not satisfiable, for any such $x$, there exists $i \in \{1,\ldots,k\}$ such that $\phi_{i1} + \phi_{i2} + \phi_{i3}+1 \neq 0$, as there must be a clause where not exactly one literal is set to one. This means that $s_{\phi}$ is positive everywhere, which is a contradiction. 
	


	
	
	We have thus shown that testing attainment of the optimal value is NP-hard for unconstrained POPs where the objective is a polynomial of degree 6. In the interest of minimality, we now extend the proof to apply to an objective function of degree 4. To do this, \jeff{we first introduce $n+1$ new variables $\chi_1, \ldots, \chi_n$ and $w$}. \jeff{We} replace every occurrence of the product $\lambda x_i$ \jeff{in $\lambda^2 s_\phi$} with the variable $\chi_i$. For example, the term $\lambda^2 x_1x_2$ would become $\chi_1\chi_2$. Let $\hat{s}_\phi(x,\chi,\lambda)$ denote this transformation on $\lambda^2s_\phi(x)$. Note that \jeff{$\hat{s}_\phi(x, \chi, \lambda)$} is now a quartic polynomial. Now consider the quartic polynomial (whose coefficients are again \jeff{at most a constant factor of $n+k$ in absolute value})
	\begin{equation}\label{eq: 1 in 3 sos new}
	\hat{p}_\phi(x,y,z,\lambda,\chi,w) = \hat{s}_\phi(x, \chi,\lambda) + (1-\lambda)^2(y^2 + (w-1)^2) + (w-yz)^2 + \sum_{i=1}^n (\chi_i - \lambda_ix_i)^2.
	\end{equation}
	Observe that $\hat{p}_\phi$ is a sum of squares as $\hat{s}_\phi$ can be verified to be a sum of squares by bringing $\lambda$ inside every squared term of $s_\phi$. Hence, $\hat{p}_\phi$ is nonnegative. Furthermore, its infimum is still zero, as the choice of variables $\lambda = 0, w = 1, \chi = 0$, $x$ arbitrary, $z = \frac{1}{y}$, and letting $y \to \infty$ will result in arbitrarily small values of $\hat{p}_\phi$. Now it remains to show that this polynomial will have a zero if and only if $p_{\phi}$ in (\ref{eq: 1 in 3 sos plus unattained}) has a zero. Observe that if $(x,y,z,\lambda)$ is a zero of $p_{\phi}$, then $(x,y,z, \lambda, \lambda x,yz)$ is a zero of $\hat{p}_{\phi}$. Conversely, if $(x,y,z,\lambda,\chi,w)$ is a zero of $\hat{p}_\phi$, then $(x,y,z,\lambda)$ is a zero of $p_{\phi}$. \end{proof}

\begin{remark} Because we use the ideas behind this reduction repeatedly in the remainder of this paper, we refer to the quartic polynomial defined in ($\ref{eq: 1 in 3 sos}$) as $s_\phi$ throughout. The same convention for $\phi_{it}$ relating the literals of $\phi$ to the variables $x$ will be assumed as well.

\end{remark}

We next show that testing attainment of the optimal value of a POP is NP-hard when the objective function is linear and the constraints are quadratic. Together with the previously-known Frank-Wolfe type theorems which we reviewed in the introduction, Theorems \ref{Thm: AOS NP-hard o4c0} and \ref{Thm: AOS NP-hard o1c2} characterize the complexity of testing attainment of the optimal value in polynomial optimization problems of any given degree. Indeed, our reductions can trivially be extended to the case where the constraints or the objective have higher degrees. \jeff{For example to increase the degree of the constraints to some positive integer $d$, one can introduce a new variable $\gamma$ along with the trivial constraint $\gamma^d = 0$. To increase the degree of the objective from four to a higher degree $2d$, one can again introduce a new variable $\gamma$ and add the term $\gamma^{2d}$ to the objective function.}


\begin{theorem}\label{Thm: AOS NP-hard o1c2}
	Testing whether a degree-1 polynomial attains its infimum on a feasible set defined by degree-2 inequalities is strongly NP-hard.
\end{theorem}
\begin{proof}
	Consider a ONE-IN-THREE 3SAT instance $\phi$ with $n$ variables and $k$ clauses. Define the following POP, with $x,\chi \in \mathbb{R}^n$ and $\lambda, y,z,w, \gamma, \zeta, \psi \in \mathbb{R}$:
	
	\begin{align}\label{Defn: ONE-IN-THREE 3SAT QCQP}
	& \underset{x, \chi, \lambda, y,z,w,\gamma, \zeta, \psi}{\min}
	& & \gamma \\
	& \text{subject to}
	&& \gamma \ge \lambda \sum_{i=1}^n \chi_i + (1-\lambda) (\psi + \zeta)\label{const: obj surrogate}\\
	&&& 1-x_i^2 = 0,\ \forall i \in \{1,\ldots,n\},\label{const: hypercube}\\
	&&&\chi_i = (\phi_{i1} + \phi_{i2} + \phi_{i3} + 1)^2,\ \forall i \in \{1,...,k\}, \label{const: 3sat sos}\\
	&&& \psi = y^2, \label{const: y2}\\
	&&& yz = w,\label{const: yz}\\
	&&& \zeta = (w-1)^2, \label{const: yz2}\\
	&&& \lambda(1-\lambda) = 0 \label{const: switch}.
	\end{align}
	
	We show that the infimum of this POP is attained if and only if $\phi$ is satisfiable. Note first that the objective value is always nonnegative because of (\ref{const: obj surrogate}) and in view of (\ref{const: 3sat sos}), (\ref{const: y2}), (\ref{const: yz2}), and (\ref{const: switch}). Observe that if $\phi$ has a satisfying assignment, then letting $x_i=1$ if the variable is true in that assignment and $x_i=-1$ if it is false, along with $\lambda = 1$, $y$ and $z$ arbitrary, $\psi = y^2, w = yz,$ and $\zeta = (w-1)^2$, results in a feasible solution with an objective value of 0. 
	
	If $\phi$ is not satisfiable, the objective value can be made arbitrarily close to zero by taking an arbitrary $x \in \{-1,1\}^n,$ $\chi_i$ accordingly to satisfy (\ref{const: 3sat sos}), $\lambda = 0, \psi = y^2, z = \frac{1}{y}, w = 1,\zeta=0$, and letting $y \to 0 $. Suppose for the sake of contradiction that there exists a feasible solution to the POP with $\gamma=0.$ As argued before, because of (\ref{const: 3sat sos}), (\ref{const: y2}), (\ref{const: yz2}), and (\ref{const: switch}), $\lambda \sum_{i=1}^n \chi_i + (1-\lambda) (\psi + \zeta)$ is always nonnegative, and so for $\gamma$ to be exactly zero, we need to have $$\lambda \sum_{i=1}^n \chi_i + (1-\lambda) (\psi + \zeta)=0.$$ From (\ref{const: switch}), either $\lambda=0$ or $\lambda=1.$ If $\lambda=1$, then we must have $\chi_i=0,\forall i=1,\ldots,n$, which is not possible as $\phi$ is not satisfiable. If $\lambda=0$, then we must have $\psi+\zeta=y^2+(yz-1)^2=0$, which cannot happen as this would require $y=0$ and $yz=1$ concurrently.	
\end{proof}

\section{NP-hardness of Testing Sufficient Conditions for Attainment}\label{Sec: Sufficient Conditions}

Arguably, the two best-known sufficient conditions under which problem (\ref{Defn: pop}) attains its optimal value are \emph{compactness} of the feasible set and \emph{coercivity} of the objective function. In this section, we show that both of these properties are NP-hard to test for POPs of low degree. We also prove that certain stronger conditions, namely the \emph{Archimedean property} of the quadratic module associated with the constraints and  \emph{stable compactness} of the feasible set, are NP-hard to test.



\subsection{Coercivity of the Objective Function}\label{SSec: Coercive NP-hard}

A function $p: \mathbb{R}^n \to \mathbb{R}$ is \emph{coercive} if for every sequence $\{x_k\}$ such that $\|x_k\| \to \infty$, we have $p(x_k) \to \infty$. It is well known that a continuous coercive function achieves its infimum on a closed set (see, e.g., Appendix A.2 of \cite{bertsekas1999nonlinear}). This is because all sublevel sets of continuous coercive functions are compact.


%

\begin{theorem}\label{Thm: Coercive NP-hard} 
Testing whether a degree-4 polynomial is coercive is strongly NP-hard.
\end{theorem}
\begin{proof}
	
Consider a ONE-IN-THREE 3SAT instance $\phi$ with $n$ variables and $k$ clauses, and the associated quartic polynomial $s_{\phi}(x)$ as in (\ref{eq: 1 in 3 sos}). Let $s_{\phi h}:\mathbb{R}^{n+1} \rightarrow \mathbb{R}$ be the homogenization of this polynomial:

	\begin{equation}\label{eq: 1 in 3 sos homo}
		s_{\phi h}(x_0, x) \defeq x_0^4 s_\phi \left(\frac{x}{x_0}\right) = \sum_{i=1}^k x_0^2(\phi_{i1} + \phi_{i2} + \phi_{i3} + x_0)^2 + \sum_{i=1}^n (x_0^2-x_i^2)^2.
		\end{equation}
	By construction, $s_{\phi h}$ is a homogeneous polynomial of degree 4. We show that $s_{\phi h}$ is coercive if and only if $\phi$ is not satisfiable.

%
%
%
	
	
	Suppose first that the instance $\phi$ has a satisfying assignment $\hat{x} \in \{-1,1\}^n$. Then it is easy to see that $s_{\phi h}(1,\hat{x}) = 0$. As $s_{\phi h}$ is homogeneous, $s_{\phi h}(\alpha,\alpha\hat{x}) = 0$ for all $\alpha$, showing that $s_{\phi h}$ is not coercive.
	
	Now suppose that $\phi$ is not satisfiable. We show that $s_{\phi h}$ is positive definite (i.e., $s_{\phi h}(x_0,x)>0$ for all $(x_0,x)\neq (0,0) $). This would then imply that $s_{\phi h}$ is coercive as 
	\begin{align*}
	s_{\phi h }(x_0,x)&=||(x_0,x)^T||^4\cdot s_{\phi h} \left( \frac{(x_0,x)}{||(x_0,x)^T||} \right)\\
	&\geq \mu||(x_0,x)^T||^4,
	\end{align*}
where $\mu>0$ is defined as the minimum of $s_{\phi h}$ on the unit sphere: $$\mu= \min_{(x_0,x)\in S^{n}}s_{\phi h}(x_0,x).$$ 
Suppose that $s_{\phi h}$ was not positive definite. Then there exists a point $(\hat{x}_0,\hat{x}) \ne (0,0)$ such that $s_{\phi h}(\hat{x}_0,\hat{x}) = 0$. First observe $\hat{x}_0$ cannot be zero due to the $(x_0 - x_i)^2$ terms in (\ref{eq: 1 in 3 sos homo}). As $\hat{x}_0 \ne 0$, then, by homogeneity, the point $(1, \frac{\hat{x}}{\hat{x}_0})$ is a zero of $s_{\phi h}$ as well. This however implies that $s_{\phi}(\hat{x})=0$, which we have previously argued (cf. the proof of Theorem \ref{Thm: AOS NP-hard o4c0}) is equivalent to satisfiability of $\phi$, hence a contradiction.
\end{proof}

We remark that the above hardness result is minimal in the degree as odd-degree polynomials are never coercive and a quadratic polynomial $x^TQx+b^Tx+c$ is coercive if and only if the matrix $Q$ is positive definite, a property that can be checked in polynomial time (e.g., by checking positivity of the leading principal minors of $Q$).

\subsection{Closedness and Boundedness of the Feasible Set}\label{SSec: Closed and Bounded NP-hard}

The well-known Bolzano-Weierstrass extreme value theorem states that the infimum of a continuous function on a compact (i.e., closed and bounded) set is attained. In this section, we show that testing closedness or boundedness of a basic semialgebraic set \jeff{defined by degree-2 inequalities} is NP-hard. Once again, these hardness results are minimal in degree \jeff{since} these properties can be tested in polynomial time for \jeff{sets defined by affine inequalities, as we describe next.
	
To check boundedness of a set $P \defeq \{x \in \mathbb{R}^n\ |\ a_i^Tx \ge b_i, i = 1,\ldots,m\}$ defined by affine inequalities, one can first check that $P$ is nonempty, and if it is, for each $i$ minimize and maximize $x_i$ over $P$. Note that $P$ is unbounded if and only if at least one of these $2n$ linear programs is unbounded, which can be certified e.g. by detecting infeasibility of the corresponding dual problem. Thus, boundedness of $P$ can be tested by solving $2n+1$ linear programming feasibility problems, which can be done in polynomial time.

To check closedness of a set $P \defeq \{x \in \mathbb{R}^n\ |\ a_i^Tx \ge b_i, i = 1, \ldots, m, c_j^Tx > d_j, j = 1, \ldots, r\}$, one can for each $j$ minimize $c_j^Tx$ over $\{x \in \mathbb{R}^n\ |\ a_i^Tx \ge b_i, i = 1, \ldots, m\}$ and declare that $P$ is closed if and only if all of the respective optimal values are greater than $d_j$. Thus, closedness of $P$ can be tested by solving $r$ linear programs, which can be done in polynomial time.}



\begin{theorem}\label{Thm: Closedness} 
Given a set of quadratic polynomials $g_i, i = 1, \ldots, m, h_j, j = 1, \ldots, r$, it is strongly NP-hard to test whether the \jeff{basic semialgebraic} set
	$$\{x \in \mathbb{R}^n |~ g_i(x) \ge 0, i = 1, \ldots, m, h_j(x) > 0, j = 1, \ldots, r\}$$
	is closed\footnote{\jeff{Note that $m$ is not fixed in this statement or in Corollaries \ref{Thm: Boundedness} and \ref{Thm: Stable Compact NP-hard} below.}}.\end{theorem}
\begin{proof}
	Consider a ONE-IN-THREE 3SAT instance $\phi$ with $n$ variables and $k$ clauses. Let $\phi_{ij}$ be as in the proof of Theorem \ref{Thm: AOS NP-hard o4c0} and consider the set
	\begin{equation}\label{eq:def.S.phi}
	S_{\phi} = \big\{(x,y)\in \mathbb{R}^{n+1}|~(\phi_{i1} + \phi_{i2} + \phi_{i3} + 1)y = 0, i=1,\ldots,k, 1 - x_j^2 = 0, j = 1, \ldots, n, y < 1\big\}.
	\end{equation}
	
	We show that $S_{\phi}$ is closed if and only if the instance $\phi$ is not satisfiable. To see this, first note that we can rewrite $S_{\phi}$ as
	$$S_{\phi} = \Big\{\{-1,1\}^n \times \{0\} \Big\} \cup \Big\{\{x \in \mathbb{R}^n|~s_\phi(x) = 0\} \times \{y\in \mathbb{R}|~y < 1\}\Big\},$$
	where $s_{\phi}$ is as in the proof of Theorem \ref{Thm: AOS NP-hard o4c0}. If $\phi$ is not satisfiable, then $S_{\phi}=\{-1,1\} ^n \times \{0\}$, which is closed. If $\phi$ is satisfiable, then $\{x \in \mathbb{R}^n|~s_\phi(x) = 0\}$ is nonempty and $$\{x \in \mathbb{R}^n|~s_\phi(x) = 0\} \times \{y\in \mathbb{R}|~y < 1\}$$ is not closed and not a subset of $\{-1,1\}^n \times \{0\}$. This implies that $S_{\phi}$ is not closed.	
\end{proof}

\begin{remark}\label{rem:boundedness.compactness}
We note that the problem of testing closedness of a basic semialgebraic set remains NP-hard even if one has a promise that the set is bounded. Indeed, one can add the constraint $y\geq -1$ to the set $S_{\phi}$ in (\ref{eq:def.S.phi}) to make it bounded and this does not change the previous proof.
\end{remark}


\begin{cor}\label{Thm: Boundedness} 
Given a set of quadratic polynomials $g_i, i = 1, \ldots, m,$ it is strongly NP-hard to test whether the set
	$$\{x \in \mathbb{R}^n |~ g_i(x) = 0, i = 1, \ldots, m\}$$
	is bounded.

\end{cor}
\begin{proof}
	Consider a ONE-IN-THREE 3SAT instance $\phi$ with $n$ variables and $k$ clauses. Let $\phi_{ij}$ be as in the proof of Theorem \ref{Thm: AOS NP-hard o4c0} and consider the set
	$$S = \big\{(x,y)\in \mathbb{R}^{n+1}|~(\phi_{i1} + \phi_{i2} + \phi_{i3} + 1)y = 0, i=1,\ldots,k, 1 - x_j^2 = 0, j = 1, \ldots, n\big\}.$$
	
	This set is bounded if and only if $\phi$ is not satisfiable. One can see this by following the proof of Theorem \ref{Thm: Closedness} and observing that $y$ will be unbounded in the satisfiable case, and only 0 otherwise.
\end{proof}

Note that it follows immediately from either of the results above that testing compactness of a basic semialgebraic set is NP-hard. We end this subsection by establishing the same hardness result for a sufficient condition for compactness that has featured in the literature on polynomial optimization (see, e.g., \cite{marshall2003optimization}, \cite[Section 7]{nie2006minimizing}).

\begin{defn}\label{Defn: Stably Compact}
	A closed basic semialgebraic set $S = \{x \in \mathbb{R}^n|~g_i(x) \ge 0, i = 1,\ldots, m\}$ is \emph{stably compact} if there exists $\epsilon > 0$ such that the set $\{x \in \mathbb{R}^n |~ \delta_i(x) + g_i(x) \ge 0,i = 1,\ldots,m\}$ is compact for any set of polynomials $\delta_i$ having degree at most that of $g_i$ and coefficients at most $\epsilon$ in absolute value.
\end{defn}

Intuitively, a \jeff{closed} basic semialgebraic set is stably compact if it remains compact under small perturbations of the coefficients of its defining polynomials. A stably compact set is clearly compact, though the converse is not true as shown by the set $$S = \big\{(x_1,x_2) \in \mathbb{R}^2|~ (x_1-x_2)^4 + (x_1+x_2)^2 \le 1\big\}.$$ Indeed, this set is contained inside the unit disk, but for $\epsilon>0,$ the set $$S_{\epsilon} = \big\{(x_1,x_2) \in \mathbb{R}^2|~ (x_1-x_2)^4 - \epsilon x_1^4 + (x_1+x_2)^2 \le 1\big\}$$ is unbounded as its defining polynomial tends to $-\infty$ along the line $x_1 = x_2$. 

Section 5 of \cite{marshall2003optimization} shows that the set $S$ in Definition \ref{Defn: Stably Compact} is stably compact if and only if the function $$q(x)=\underset{i,j}\max\{-g_{ij}(x)\}$$ is positive on the unit sphere. Here, $g_{ij}(x)$ is a homogenenous polynomial that contains all terms of degree $j$ in $g_i(x).$ Perhaps because of this characterization, the same section in \cite{marshall2003optimization} remarks that ``stable compactness is easier to check than compactness'', though as far as polynomial-time checkability is concerned, we show that the situation is no better.



\begin{cor}\label{Thm: Stable Compact NP-hard} 
Given a set of quadratic polynomials $g_i, i = 1, \ldots, m,$ it is strongly NP-hard to test whether the set
	$$\{x \in \mathbb{R}^n | ~g_i(x) = 0, i = 1, \ldots, m\}$$
	is stably compact.
%
\end{cor}

\begin{proof}
	Consider a ONE-IN-THREE 3SAT instance $\phi$ with $n$ variables and $k$ clauses. Let $\phi_{ij}$ be as in the proof of Theorem \ref{Thm: AOS NP-hard o4c0} and consider the set
	$$T_{\phi} = \big\{(x_0,x)\in \mathbb{R}^{n+1}|~(\phi_{i1} + \phi_{i2} + \phi_{i3} + x_0)^2 = 0, i=1,\ldots,k, x_0^2 - x_j^2 = 0, j = 1, \ldots, n\big\}.$$
	We show that the function $$q_{\phi}(x_0,x)=\max_{i=1,\ldots,k, j=1,\ldots,n}\{-(\phi_{i1} + \phi_{i2} + \phi_{i3} + x_0)^2, (\phi_{i1} + \phi_{i2} + \phi_{i3} + x_0)^2, x_0^2 - x_j^2, x_j^2-x_0^2\}$$ is positive on the unit sphere if and only if the instance $\phi$ is not satisfiable. Suppose first that $\phi$ is not satisfiable and assume for the sake of contradiction that there is a point $(x_0,x)$ on the sphere such that $q_{\phi}(x_0,x)=0.$ This implies that $\phi_{i1} + \phi_{i2} + \phi_{i3} + x_0=0, \forall i=1,\ldots,k$ and $x_0^2=x_j^2,\forall j=1,\ldots,n.$ Hence, $x_0 \neq 0$ and $\frac{x}{x_0}$ is a satisfying assignment to $\phi$, which is a contradiction. 
	Suppose now that $\phi$ has a satisfying assignment $\hat{x} \in \{-1,1\}^n.$ Then it is easy to check that 	$$q_{\phi}\left(\frac{(1,\hat{x})}{||(1,\hat{x})^T||}\right)=0.$$
	%
%
%
%
%
%
%
%
\end{proof}

\subsubsection{The Archimedean Property}\label{SSSec: Archimedean}

An algebraic notion closely related to compactness is the so-called Archimedean property. This notion has frequently appeared in recent literature at the interface of algebraic geometry and polynomial optimization. The Archimedean property is the assumption needed for the statement of Putinar's Positivstellensatz \cite{putinar1993positive} and convergence of the Lasserre hierarchy \cite{lasserre2001global}. In this subsection, we recall the definition of the Archimedean property and study the complexity of checking it. To our knowledge, the only previous result in this direction is that testing the Archimedean property is decidable \cite[Section 3.3]{wagner2009archimedean}.

We say that a polynomial $p$ is a \emph{sum of squares} (sos) if there exist polynomials $q_1,\ldots,q_r$ such that $p= \sum_{i=1}^r q_i^2$. An sos polynomial is clearly always nonnegative. The \emph{quadratic module} associated with a set of polynomials $g_1,\ldots,g_m$ is the set of polynomials that can be written as $$\sigma_0(x)+\sum_{i=1}^m \sigma_i(x) g_i(x),$$ where $\sigma_0,\ldots,\sigma_m$ are sum of squares polynomials.


\begin{defn}\label{Defn: Archimedean}
	A quadratic module $Q$ is \emph{Archimedean} if there exists a scalar $R > 0$ such that $R - \sum_{i=1}^n x_i^2 \in Q$.
\end{defn}

Several equivalent characterizations of this property can be found in \cite[Theorem 3.17]{laurent2009sums}. Note that a set $\{x \in \mathbb{R}^n|~g_i(x) \ge 0\}$ for which the quadratic module associated with the polynomials $\{g_i\}$ is Archimedean is compact. However, the converse is not true. For example, for $n>1$, the sets $$\left\{x \in \mathbb{R}^n|~x_1-\frac{1}{2} \ge 0, \ldots, x_n - \frac{1}{2} \ge 0, 1-\prod_{i=1}^n x_i \ge 0\right\}$$ are compact but not Archimedean; see \cite{laurent2009sums}, \cite{prestel2001positive} for a proof of the latter claim. Hence, hardness of testing the Archimedean property does not follow from hardness of testing compactness. 


%

As mentioned previously, the Archimedean property has received recent attention in the optimization community due to its connection to the Lasserre hierarchy. Indeed, under the assumption that the quadratic module associated with the defining polynomials of the feasible set of (\ref{Defn: pop}) is Archimedean, the Lasserre hierarchy \cite{lasserre2001global} produces a sequence of SDP-based lower bounds that converge to the optimal value of the POP. Moreover, Nie has shown~\cite{nie2014optimality} that under the Archimedean assumption, convergence happens in a finite number of rounds generically. One way to ensure the Archimedean property---assuming that we know that our feasible set is contained in a ball of radius $R$--- is to add the redundant constraint $R^2 \geq \sum_{i=1}^n x_i^2$ to the constraints of (\ref{Defn: pop}). This approach however increases the size of the SDP that needs to be solved at each level of the hierarchy. Moreover, such a scalar $R$ may not be readily available for some applications.


Our proof of NP-hardness of testing the Archimedean property will be based on showing that the specific sets that arise from the proof of Corollary \ref{Thm: Boundedness} are compact if and only if their corresponding quadratic modules are Archimedean. Our proof technique will use the Stengle's Positivstellensatz, which we recall next.

\begin{theorem}[Stengle's Positivstellensatz \cite{stengle1974nullstellensatz}]\label{Thm: Stengle}
A basic semialgebraic set $$\mathcal{S} \defeq \{x\in\mathbb{R}^n|~ g_i(x) \ge 0, i= 1,\ldots, m, h_j(x) = 0, j = 1,\ldots, k\}$$ is empty if and only if there exist sos polynomials $\sigma_{c_1, \ldots, c_m}$ and polynomials $t_i$ such that
$$-1 = \sum_{j=1}^k t_jh_j + \sum_{c_1,\ldots,c_m \in \{0,1\}^m} \sigma_{c_1,\ldots,c_m}(x) \Pi_{i=1}^m g_i(x)^{c_i}.$$
\end{theorem}

\begin{remark} Note that if only equality constraints are considered, the second term on the right hand side is a single sos polynomial $\sigma_{0,\ldots, 0}$. \jeff{In the next theorem, we only need this special case, which is also known as the Real Nullstellensatz \cite{krivine1964anneaux}.}
\end{remark}

\begin{theorem}\label{Thm: Archimedean NP-hard} 
Given a set of quadratic polynomials $g_1,\ldots,g_m$, it is strongly NP-hard to test whether their quadratic module has the Archimedean property.

\end{theorem}

\begin{proof}
Consider a ONE-IN-THREE 3SAT instance $\phi$ with $n$ variables and $k$ clauses. Let $\phi_{ij}$ be as in the proof of Theorem \ref{Thm: AOS NP-hard o4c0} and consider the set of quadratic polynomials
$$\big\{(\phi_{i1}+\phi_{i2}+\phi_{i3})y, -(\phi_{i1}+\phi_{i2}+\phi_{i3})y, i=1,\ldots,k; 1-x_j^2, x_j^2-1, j=1,\ldots,n\big\}.$$


We show that the quadratic module associated with these polynomials is Archimedean if and only if $\phi$ is not satisfiable. First observe that if $\phi$ is satisfiable, then the quadratic module cannot be Archimedean as the set 
$$S = \big\{(x,y)\in \mathbb{R}^{n+1}|~(\phi_{i1} + \phi_{i2} + \phi_{i3} + 1)y = 0, i=1,\ldots,k, 1 - x_j^2 = 0, j = 1, \ldots, n\big\}$$
is not compact (see the proof of Corollary \ref{Thm: Boundedness}).

Now suppose that the instance $\phi$ is not satisfiable. 
We need to show that for some scalar $R > 0$ and some sos polynomials $\sigma_0,\sigma_1,\ldots,\sigma_k,$ $\hat{\sigma}_1,\ldots,\hat{\sigma}_k$,$\tau_1,\ldots,\tau_n$,$\hat{\tau}_1,\ldots,\hat{\tau}_n$, we have
\begin{align*}
R-\sum_{i=1}^n x_i^2-y^2 &=\sigma_0(x,y)+\sum_{i=1}^k \sigma_i(x,y) (\phi_{i1}+\phi_{i2}+\phi_{i3}+1)y\\
&+\sum_{i=1}^k \hat{\sigma}_i(x,y) (-\phi_{i1}-\phi_{i2}-\phi_{i3}-1)y
+\sum_{j=1}^n \tau_j(x,y) (1-x_j^2)+\sum_{j=1}^n\hat{\tau}_j(x,y)(x_j^2-1).
\end{align*}
Since any polynomial can be written as the difference of two sos polynomials (see, e.g., \cite[Lemma 1]{ahmadi2015dc}), this is equivalent to existence of a scalar $R>0$, an sos polynomial $\sigma_0$, and some polynomials $v_1,\ldots,v_k, t_1,\ldots,t_n$ such that 
\begin{align}\label{eq:archimedean}
R-\sum_{i=1}^n x_i^2-y^2 &=\sigma_0(x,y)+\sum_{i=1}^k v_i(x,y) (\phi_{i1}+\phi_{i2}+\phi_{i3}+1)y +\sum_{j=1}^n t_j(x,y) (1-x_j^2).
\end{align}

First, note that 
\begin{align}\label{eq:first.part.arch}
n-\sum_{i=1}^n x_i^2=\sum_{i=1}^n (1-x_i^2).
\end{align}
Secondly, as $\phi$ is not satisfiable, we know that the set 
$$\{x \in \mathbb{R}^n|~1-x_j^2 = 0, j= 1,\ldots,n, \phi_{i1} + \phi_{i2} + \phi_{i3} + 1 = 0,i = 1, \ldots, k\}$$
is empty. From Stengle's Positivstellensatz, it follows that there exist an sos polynomial $\tilde{\sigma}_0$ and some polynomials $\tilde{v}_1,\ldots, \tilde{v}_k$, $\tilde{t}_1,\ldots,\tilde{t}_n$ such that 
\begin{align*}
-1=\tilde{\sigma}_0(x)+\sum_{i=1}^k \tilde{v}_i(x)(\phi_{i1}+\phi_{i2}+\phi_{i3}+1)+\sum_{j=1}^n \tilde{t}_j(x)(1-x_j^2).
\end{align*}
Multiplying this identity on either side by $y^2$, we obtain:
\begin{align}\label{eq:second.part.arch}
-y^2=y^2\tilde{\sigma}_0(x)+\sum_{i=1}^k \tilde{v}_i(x)y \cdot(\phi_{i1}+\phi_{i2}+\phi_{i3}+1)y+\sum_{j=1}^n \tilde{t}_j(x)y^2(1-x_j^2).
\end{align}
Note that if we sum (\ref{eq:first.part.arch}) and (\ref{eq:second.part.arch}) and take $R=n$, $\sigma_0(x,y)=y^2\tilde{\sigma}_{0}(x)$, $v_i(x,y)=y\tilde{v}_i(x)$ for all $i=1,\ldots,k$, and $t_j(x,y)=y^2 \cdot \tilde{t}_j(x)+1$ for all $j=1,\ldots,n$, we recover (\ref{eq:archimedean}).

\end{proof}

%
%
%

\section{Algorithms for Testing Attainment of the Optimal Value}\label{Sec: Algorithms}
In this section, we give a hierarchy of sufficient conditions for compactness of a closed basic semialgebraic set, and a hierarchy of sufficient conditions for coercivity of a polynomial. These hierarchies are amenable to semidefinite programming (SDP) as they all involve, in one way or another, a search over the set of sum of squares polynomials. The connection between SDP and sos polynomials is well known: 
a polynomial $\sigma$ of degree $2d$ is sos if and only if there exists a symmetric positive semidefinite matrix $Q$ such that $\sigma(x)=z(x)^TQz(x)$ for all $x$, where $z(x)$ here is the standard vector of monomials of degree up to $d$ in the variables $x$ (see, e.g. \cite{parrilo2003semidefinite}).


The hierarchies that we present are such that if the property in question (i.e., compactness or coercivity) is satisfied on an input instance, then some level of the SDP hierarchy will be feasible and provide a certificate that the property is satisfied. The test for compactness is a straightforward application of Stengle's Positivstellensatz, but the test for coercivity requires a new characterization of this property, which we give in Theorem \ref{Thm: Polynomial Radius}.

\subsection{Compactness of the Feasible Set}\label{SSec: Compactness Alg}

Consider a closed basic semialgebraic set $$\mathcal{S} \defeq \{x \in \mathbb{R}^n | ~g_i(x) \ge 0, i = 1, \ldots, m\},$$ where the polynomials $g_i$ have integer coefficients\footnote{If some of the coefficients of the polynomials $g_i$ are rational (but not integer), we can make them integers by clearing denominators without changing the set $\mathcal{S}$.} and are of degree at most $d$. A result of Basu and Roy \cite[Theorem 3]{basu2010bounding} implies that if $\mathcal{S}$ is bounded, then it must be contained in a ball of radius 
\begin{equation}\label{Eq: Basu Bound}
R^* \defeq \sqrt{n}\left((2d+1)(2d)^{n-1}+1\right)2^{(2d+1)(2d)^{n-1}(2nd+2)\left(2\tau + bit((2d+1)(2d)^{n-1})+(n+1)bit(d+1)+bit(m)\right)},
\end{equation}
where $\tau$ is the largest bitsize of any coefficient of any $g_i$, and bit($\eta$) denotes the bitsize of $\eta$.


%


With this result in mind, the following proposition is an immediate consequence of Stengle's Positivstellensatz (c.f. Theorem \ref{Thm: Stengle}) after noting that the set $\mathcal{S}$ is bounded if and only if the set $$\{x\in \mathbb{R}^n|~g_i(x) \ge 0, i = 1, \ldots, m, \sum_{i=1}^n x_i^2 \ge R^* + 1\}$$ is empty.
\begin{prop}\label{Prop: Compactness Stengle}
	Consider a closed basic semialgebraic set $\mathcal{S} \defeq \{x \in \mathbb{R}^n |~ g_i(x) \ge 0, i = 1, \ldots, m\}$, where the polynomials $g_i$ have integer coefficients and are of degree at most $d$. Let $R^*$ be as in (\ref{Eq: Basu Bound}) and let $g_0(x) = \sum_{i=1}^n x_i^2 - R^* -1$. Then the set $\mathcal{S}$ is compact if and only if there exist sos polynomials $\sigma_{h_0,\ldots, h_m}$ such that
	$$-1 = \sum_{h_0,...,h_m \in \{0,1\}^{m+1}} \sigma_{h_0,...,h_m}(x) \Pi_{i=0}^m g_i(x)^{h_i}.$$
\end{prop}
This proposition naturally yields the following semidefinite programming-based hierarchy indexed by a nonnegative integer $r$:

\begin{equation}\label{Eq: Compactness SDP}
\begin{aligned}
& \underset{\sigma_{h_0, \ldots, h_m}}{\min}
& & 0 \\
& \text{subject to}
&& -1 = \sum_{h_0,\ldots,h_m \in \{0,1\}^{m+1}} \sigma_{h_0,\ldots,h_m}(x) \Pi_{i=0}^m g_i(x)^{h_i},\\
&&&\sigma_{h_0, \ldots, h_m} \text{ is sos and has degree }\leq 2r.\\
\end{aligned}
\end{equation}

Note that for a fixed level $r$, one is solving a semidefinite program whose size is polynomial in the description of $\mathcal{S}$. If for some $r$ the SDP is feasible, then we have an algebraic certificate of compactness of the set $\mathcal{S}.$ Conversely, as Proposition \ref{Prop: Compactness Stengle} implies, if $\mathcal{S}$ is compact, then the above SDP will be feasible for some level $r^*$. One can upper bound $r^*$ by a function of $n, m,$ and $d$ only using the main theorem of \cite{lombardi2014elementary}. This bound is however very large and mainly of theoretical interest.
%

\subsection{Coercivity of the Objective Function}\label{SSec: Coercivity Alg}

It is well known that the infimum of a continuous coercive function over a closed set is attained. This property has been widely studied, even in the case of polynomial functions; see e.g. \cite{bajbar2015coercive,bajbar2017fast,jeyakumar2014polynomial}. 
A simple sufficient condition for coercivity of a polynomial $p$ is for its terms of highest order to form a positive definite (homogeneous) polynomial; see, e.g., \cite[Lemma 4.1]{jeyakumar2014polynomial}. One can give a hierarchy of SDPs to check for this condition as is done in \cite[Section 4.2]{jeyakumar2014polynomial}. However, this condition is sufficient but not necessary for coercivity. For example, the polynomial $x_1^4 + x_2^2$ is coercive, but its top homogeneous component is not positive definite. Theorem \ref{Thm: Polynomial Radius} below gives a necessary and sufficient condition for a polynomial to be coercive which lends itself again to an SDP hierarchy. To start, we need the following proposition, whose proof is straightforward and thus omitted.

\begin{prop}\label{Prop: Coercive sublevel sets bounded}
	A function $f:\mathbb{R}^n \rightarrow \mathbb{R}$ is coercive if and only if the sets $$\mathcal{S}_\gamma \defeq \{x \in \mathbb{R}^n|~\jeff{f}(x) \le \gamma\}$$ are bounded for all $\gamma \in \mathbb{R}$.
\end{prop}

A polynomial $p$ is said to be \emph{$s$-coercive} if $p(x)/\|x\|^s$ is coercive. The \emph{order of coercivity} of $p$ is the supremum over $s \ge 0$ for which $p$ is $s$-coercive. It is known that the order of coercivity of a coercive polynomial is always positive \cite{bajbar2017fast,gorin1961asymptotic}.

\begin{theorem}\label{Thm: Polynomial Radius}
	A polynomial $p$ is coercive if and only if there exist an even integer $c > 0$ and a scalar $k \ge 0$ such that for all $\gamma \in \mathbb{R}$, the $\gamma$-sublevel set of $p$ is contained within a ball of radius $\gamma^c+k$.
\end{theorem}

\begin{proof}
	
	The ``if'' direction follows immediately from the fact that each $\gamma$-sublevel set is bounded. For the converse, suppose that $p$ is coercive and denote its order of coercivity by $q > 0$. Then, from Observation 2 of \cite{bajbar2017fast}, we get that there exists a scalar $M \ge 0$ such that 
	\begin{equation}\label{Eq: M q-coercive}
	\|x\| \ge M \Rightarrow p(x) \ge \|x\|^q,
	\end{equation} or equivalently $\|x\| \le p(x)^\frac{1}{q}$. Now consider the function $R_p: \mathbb{R} \to \mathbb{R}$ which is defined as $$R_p(\gamma) \defeq \underset{p(x) \le \gamma}{\max} \|x\|,$$ i.e. the radius of the $\gamma$-sublevel set of $p$. We note two relevant properties of this function.
	
	\begin{itemize}
		\item The function $R_p(\gamma)$ is nondecreasing. This is because the $\gamma$-sublevel set of $p$ is a subset of the $(\gamma+\epsilon)$-sublevel set of $p$ for any $\epsilon > 0$.
		\item Let $m = \inf \{\gamma| ~R_p(\gamma) \ge M\}$. We claim that \begin{equation}\label{Eq: m ge Mq}
		m \ge M^q.
		\end{equation}
		Suppose for the sake of contradiction that we had $m < M^q$. By the definition of $m$, there exists $\bar{\gamma} \in (m,M^q)$ such that $R_p(\bar{\gamma}) \ge M$. This means that there exists $\bar{x} \in \mathbb{R}^n$ such that $p(\bar{x}) \le \bar{\gamma} < M^q$ and $\|\bar{x}\| \ge M$. From (\ref{Eq: M q-coercive}) we then have $p(\bar{x}) \ge \|\bar{x}\|^q \ge M^q$, which is a contradiction.
	\end{itemize}
	
	We now claim that $R_p(\gamma) \le \gamma^\frac{1}{q}$ for all $\gamma > m$. Suppose for the sake of contradiction that there exists $\gamma_0 > m$ such that $R_p(\gamma_0) > \gamma_0^\frac{1}{q}$. This means that there exists $x_0 \in \mathbb{R}^n$ such that $\|x_0\| > \gamma_0^\frac{1}{q}$ but $p(x_0) \le \gamma_0$.
	
	Consider first the case where $p(x_0) \ge m$. Since $\gamma_0 > m$, we have
	$$\|x_0\| > \gamma_0^{1/q} > m^{1/q} \ge M,$$
	where the last inequality follows from (\ref{Eq: m ge Mq}). It follows from (\ref{Eq: M q-coercive}) that $p(x_0) \ge \|x_0\|^q > \gamma_0$ which is a contradiction.
	
	Now consider the case where $p(x_0) < m$. By definition of $m$, we have $R_p(p(x_0)) < M$, and so $\|x_0\| < M$. Furthermore, since $$\gamma_0 > m \overset{(\ref{Eq: m ge Mq})}{\ge} M^q,$$ we have $M < \gamma_0^{1/q}$, which gives $\|x_0\| < M < \gamma_0^{1/q}$. This contradicts our previous assumption that $\|x_0\| > \gamma_0^{1/q}$.
	
	If we let $c$ be the smallest even integer greater than $1/q$, we have shown that $$R_p(\gamma) \le \gamma^{1/q} \le \gamma^c$$ on the set $\gamma > m$. Finally, if we let $k = R_p(m)$, by monotonicity of $R_p$, we get that $R_p(\gamma) \le \gamma^c + k$, for all $\gamma$.
	
\end{proof}

\begin{remark}\label{Rem: Polynomial R2}
	One can easily show now that for any coercive polynomial $p$, there exist an integer $c' > 0$ and a scalar $k' \ge 0$ (possibly differing from the scalars $c$ and $k$ given in the proof of Theorem \ref{Thm: Polynomial Radius}) such that $$R_p^2(\gamma) < \gamma^{2c'} + k'.$$ For the following hierarchy it will be easier to work with this form.
\end{remark}

In view of the above remark, observe that coercivity of a polynomial $p$ is equivalent to existence of an integer $c' > 0$ and a scalar $k' \ge 0$ such that the set \begin{equation}\label{Eq: Coercivity Set}\left\{(\gamma, x)\in \mathbb{R}^{n+1}|~p(x) \le \gamma, \sum_{i=1}^n x_i^2 \ge \gamma^{2c'} + k'\right\}\end{equation} is empty. This formulation naturally leads to the following SDP hierarchy indexed by a positive integer $r$.

\begin{prop}\label{prop:coercivity.SDP.hierarchy}
A polynomial $p$ of degree $d$ is coercive if and only if for some integer $r \geq 1$, the following SDP is feasible:

\begin{equation}\label{Eq: Coercivity SDP}
\begin{aligned}
& \underset{\sigma_0, \ldots, \sigma_3}{\min}
& & 0 \\
& \emph{subject to}
& -1 =& \sigma_0(x,\gamma) + \sigma_1(x,\gamma)(\gamma - p(x)) + \sigma_2(x,\gamma)\left(\sum_{i=1}^n x_i^2 - \gamma^{2r} - 2^r\right) \\
&&&+ \sigma_3(x,\gamma)(\gamma - p(x))\left(\sum_{i=1}^n x_i^2 - \gamma^{2r} - 2^r\right),\\
&&& \sigma_0 \emph{ is sos and of degree } \leq 4r,\\
&&& \sigma_1 \emph{ is sos and of degree } \leq \max\{4r-d,0\},\\
&&& \sigma_2 \emph{ is sos and of degree } \leq 2r,\\
&&& \sigma_3 \emph{ is sos and of degree } \leq \max\{2r-d,0\}.\\
\end{aligned}
\end{equation}
\end{prop}

\begin{proof}
	If the SDP in (\ref{Eq: Coercivity SDP}) is feasible for some $r$, then the set
	\begin{equation}\label{Eq: Coercive Empty Set}
	\left\{(\gamma, x) \in \mathbb{R}^{n+1}|~p(x) \le \gamma, \sum_{i=1}^n x_i^2 \ge \gamma^{2r} + 2^r\right\}
	\end{equation}
	must be empty. Indeed, if this was not the case, a feasible $(\gamma,x)$ pair would make the right hand side of the equality constraint of (\ref{Eq: Coercivity SDP}) nonnegative, while the left hand side is negative. As the set in (\ref{Eq: Coercive Empty Set}) is empty, then for all $\gamma$, the $\gamma$-sublevel set of $p$ is contained within a ball of radius $\sqrt{\gamma^{2r} + 2^r}$ and thus $p$ is coercive.
	
	To show the converse, suppose that $p$ is coercive. Then we know from  Theorem \ref{Thm: Polynomial Radius} and Remark \ref{Rem: Polynomial R2} that there exist an integer $c' > 0$ and a scalar $k' \ge 0$ such that the set in (\ref{Eq: Coercivity Set}) is empty. From Stengle's Positivstellensatz (c.f. Theorem \ref{Thm: Stengle}), there exist an even nonnegative integer $\hat{r}$ and sos polynomials $\sigma_0', \ldots, \sigma_3'$ of degree at most $\hat{r}$ such that
	\begin{equation}\label{eq:Stengle.coercive}
	\begin{aligned}
	-1 =& \sigma_0'(x,\gamma) + \sigma_1'(x,\gamma)(\gamma - p(x)) + \sigma_2'(x,\gamma)\left(\sum_{i=1}^n x_i^2 - \gamma^{2c'} - k'\right) \\
	&+ \sigma_3'(x,\gamma)(\gamma - p(x))\left(\sum_{i=1}^n x_i^2 - \gamma^{2c'} - k'\right).
	\end{aligned}	
	\end{equation}
	Let $r^*=\lceil \max\{c', \log_2(k'+1), \frac{\hat{r}+d}{2}\}\rceil$. We show that the SDP in (\ref{Eq: Coercivity SDP}) is feasible for $r=r^*$ by showing that the polynomials
	$$\sigma_0(x,\gamma) = \sigma_0'(x,\gamma) + \sigma_2'(x,y)(\gamma^{2r^*}-\gamma^{2c'} + 2^{r^*}-k'),$$
	$$\sigma_1(x,\gamma) = \sigma_1'(x,\gamma)+\sigma_3'(x,\gamma)(\gamma^{2r^*}-\gamma^{2c'}+ 2^{r^*}-k'),$$
	$$\sigma_2(x,\gamma) = \sigma_2'(x,\gamma),$$
	$$\sigma_3(x,\gamma) = \sigma_3'(x,\gamma)$$
	are a feasible solution to the problem. First, note that $$4r^*-d\geq 2r^*-d\geq \hat{r}\geq 0,$$ and hence $\sigma_0$ is of degree at most $\hat{r}+2r^* \leq 4r^*$, $\sigma_1$ is of degree at most $\hat{r}+2r^*\leq \max\{4r^*-d,0\}$, $\sigma_2$ is of degree at most $\hat{r}\leq 2r^*$, and $\sigma_3$ is of degree at most $\hat{r}\leq \max\{2r^*-d,0\}$. Furthermore, these polynomials are sums of squares. To see this, note that $\gamma^{2r^*}-\gamma^{2c'}+2^{r^*}-k'$ is nonnegative as $r^*\geq c'$ and $2^{r^*}\geq k'+1$. As any nonnegative univariate polynomial is a sum of squares (see, e.g., \cite{blekherman2012semidefinite}), it follows that $\gamma^{2r^*}-\gamma^{2c'}+ 2^{r^*}-k'$ is a sum of squares. Combining this with the facts that $\sigma_0',\ldots,\sigma_3'$ are sums of squares, and products and sums of sos polynomials are sos again, we get that $\sigma_0,\ldots,\sigma_3$ are sos. Finally, the identity
	\begin{align*}
	-1 =& \left(\sigma_0'(x,\gamma)+ \sigma_2'(x,y)(\gamma^{2r^*}-\gamma^{2c'} + 2^{r^*}-k')\right)\\
	&+ \left(\sigma_1'(x,\gamma)+\sigma_3'(x,\gamma)(\gamma^{2r^*}-\gamma^{2c'}+ 2^{r^*}-k'))(\gamma - p(x)\right) \\
	&+ \sigma_2'(x,\gamma)(\sum_{i=1}^n x_i^2 - \gamma^{2r^*} - 2^{r^*}) + \sigma_3'(x,\gamma)(\gamma - p(x))(\sum_{i=1}^n x_i^2 - \gamma^{2r^*} - 2^{r^*})
	\end{align*}
	holds by a simple rewriting of (\ref{eq:Stengle.coercive}).
	
\end{proof}

As an illustration, we revisit the simple example $p(x) = x_1^4 + x_2^2$, whose top homogeneous component is not positive definite. 
The hierarchy in Proposition \ref{prop:coercivity.SDP.hierarchy} with $r=1$ gives an automated algebraic proof of coercivity of $p$ in terms of the following identity:
\begin{equation}\label{Eq: Coercivity Hierarchy Example}
-1 = \left(\frac{2}{3}(x_1^2 - \frac{1}{2})^2 + \frac{2}{3}(\gamma-\frac{1}{2})^2\right) + \frac{2}{3}(\gamma - x_1^4 - x_2^2) + \frac{2}{3}(x_1^2 + x_2^2 - \gamma^2 - 2).
\end{equation}

Note that this is a certificate that the $\gamma$-sublevel set of $p$ is contained in a ball of radius $\sqrt{\gamma^2+2}$.

\begin{remark}
From a theoretical perspective, our developments so far show that coercivity of multivariate polynomials is a decidable property as it can be checked by solving a finite number of SDP feasibility problems (each of which can be done in finite time \cite{porkolab1997complexity}). Indeed, given a polynomial $p$, one can think of running two programs in parallel. The first one solves the SDPs in Proposition~\ref{prop:coercivity.SDP.hierarchy} for increasing values of $r$. The second uses Proposition~\ref{Prop: Compactness Stengle} and its degree bound to test whether the $\beta$-sublevel set $p$ is compact, starting from $\beta=1$, and doubling $\beta$ in each iteration. On every input polynomial $p$ whose coercivity is in question, either the first program halts with a yes answer or the second program halts with a no answer. We stress that this remark is of theoretical interest only, as the value of our contribution is really in providing proofs of coercivity, not proofs of non-coercivity. Moreover, coercivity can alternatively be decided in finite time by applying the quantifier elimination theory of Tarski and Seidenberg \cite{tarski1951decision,seidenberg1954new} to the characterization in Proposition~\ref{Prop: Coercive sublevel sets bounded}.
\end{remark}

\section{Summary and Conclusions} \label{Sec:conclusion}

We studied the complexity of checking existence of optimal solutions in mathematical programs (given as minimization problems) that are feasible and lower bounded. We showed that unless P=NP, this decision problem does not have a polynomial time (or even pseudo-polynomial time) algorithm when the constraints and the objective function are defined by polynomials of low degree. More precisely, this claim holds if the constraints are defined by quadratic polynomials (and the objective has degree as low as one) or if the objective function is a quartic polynomial (even in absence of any constraints). For polynomial optimization problems with linear constraints and objective function of degrees 1,2, or 3, previous results imply that feasibility and lower boundedness always guarantee existence of an optimal solution.

We also showed, again for low-degree polynomial optimization problems, that several well-known sufficient conditions for existence of optimal solutions are NP-hard to check. These were coercivity of the objective function, closedness of the feasible set (even when bounded), boundedness of the feasible set (even when closed), an algebraic certificate of compactness known as the Archimedean property, and a robust analogue of compactness known as stable compactness.

Our negative results should by no means deter researchers from studying algorithms that can efficiently check existence of optimal solutions---or, for that matter, any of the other properties mentioned above such as compactness and coercivity---on special instances. On the contrary, our results shed light on the \jeff{intricacies} that can arise when studying these properties and calibrate the expectations of an algorithm designer. Hopefully, they will even motivate further research in identifying problem structures \jeff{(e.g., based on the Newton polytope of the objective and/or constraints)} for which checking these properties becomes more tractable, or efficient algorithms that can test useful sufficient conditions that imply these properties.
%
%
%
%
%
%

In the latter direction, we argued that sum of squares techniques could be a natural tool for certifying compactness of basic semialgebraic sets via semidefinite programming. By deriving a new characterization of coercive polynomials, we showed that the same statement also applies to the task of certifying coercivity. This final contribution motivates a problem that we leave for our future research. While coercivity (i.e., boundedness of all sublevel sets) of a polynomial objective function guarantees existence of optimal solutions to a feasible POP, the same guarantee can be made from the weaker requirement that \emph{some} sublevel set of the objective be bounded and have a non-empty intersection with the feasible set. It is not difficult to show that this property is also NP-hard to check. However, it would be useful to derive a hierarchy of sufficient conditions for it, where each level can be efficiently tested (perhaps again via SDP), and such that if the property was satisfied, then a level of the hierarchy would hold.\\

{\bf Acknowledgements:} We are grateful to Georgina Hall and two anonymous referees for their careful reading of this manuscript and very constructive feedback. We also thank Etienne de Klerk for offering a simplification of the construction in the proof of Theorem \ref{Thm: AOS NP-hard o4c0}.



%
%



\bibliographystyle{abbrv}
\bibliography{AOS_refs}

\appendix



\end{document}